\documentclass[11pt]{article}
\usepackage{latexsym}
\usepackage{amssymb,amsmath}
\usepackage{epsfig}
\usepackage{amsfonts}
\usepackage{graphicx}
\usepackage{amsmath,amsthm,amsfonts,amssymb}

\newcommand{\cp}{\,\square\,}

\newtheorem{theorem}{Theorem}

\begin{document}

\title{Distribution of Global Defensive $k$-Alliances\\ over some Graph Products}

\author{Mostafa Tavakoli $^{a}$  \and Sandi Klav\v zar $^{b,c,d}$} 

\date{}

\maketitle

\begin{center}
$^a$ Department of Applied Mathematics, Faculty of Mathematical Sciences,\\
Ferdowsi University of Mashhad, P.O.\ Box 1159, Mashhad 91775, Iran\\
{\tt m$\_$tavakoli@um.ac.ir}

\medskip

$^b$ Faculty of Mathematics and Physics, University of Ljubljana, Slovenia\\
{\tt sandi.klavzar@fmf.uni-lj.si}

\medskip

$^c$ Faculty of Natural Sciences and Mathematics, University of Maribor, Slovenia\\
\medskip

$^d$ Institute of Mathematics, Physics and Mechanics, Ljubljana, Slovenia\\
\end{center}

\begin{abstract}
If $G=(V_G, E_G)$ is a graph, then $S\subseteq V_G$ is a global defensive $k$-alliance in $G$ 
if (i) each vertex not in $S$ has a neighbor in $S$ and (ii) each vertex of $S$ has at least $k$ more neighbors inside $S$ than outside of it. The global defensive $k$-alliance number of $G$ is the minimum cardinality among all global defensive $k$-alliance in $G$. In this paper this concept is studied on the generalized hierarchical, the lexicographic, the corona, and the edge corona product. For all of these products upper bounds expressed with related invariants of the factors are given. Sharpness of the bounds are also discussed. 
\end{abstract}


\medskip\noindent
{\bf AMS Subj.\ Class:} 05C69, 05C76

\section{Introduction}
\label{sec:intro}

Alliances form a phenomena in many ways, say in politics, in relations between people (alliances such as common friendship), in social networks (say, Twitter users following each other), in natural sciences (say, animals from the same group), to list just a very small sample of examples. Since graphs are standard mathematical models for several of such phenomena, there is a strong need for a theory of alliances in graphs. The foundation for this theory was set up in~\cite{kri}, where alliances were classified into defensive, offensive, and powerful. 

In this paper we are interested in defensive alliances, the theory of which has been recently surveyed in~\cite{yero-2017}. More precisely, we are interested in global defensive $k$-alliances, a concept that widely generalizes global defensive alliances and goes back to the paper~\cite{appl3} published two years before~\cite{kri}. We refer to~\cite{Rodrıguez, yero-2011} for several mathematical properties of global defensive $k$-alliances and to~\cite{appl1,kri,appl3} for applications of alliances in as different areas as national defence, studies of RNA structures, and fault-tolerant computing.

Since determining an optimal (global) defensive ($k$-)alliance is NP-hard, a way to approach the problem is via dynamic programming: decompose a given graph into smaller parts, solve the problem on these smaller graphs, and deduce a solution or an approximate the original problem from the obtained partial solutions. Graph products and similar operations are natural candidates for such an approach. In~\cite{Eballe}, the global defensive alliance was investigated on the join, the corona, and the composition (alias lexicographic product) of graphs. Very recently, the global defensive $k$-alliance was studied on the Cartesian product, the strong product, and the direct product of graphs~\cite{y1,y2}. In this paper we continue this direction of research by investigating the global defensive $k$-alliance on additional graph products and graph operations, in particular extending some previous results. 

In the rest of the introduction concepts and notation needed are introduced, in particular the global defensive $k$-alliance is formally defined. In the subsequent section we study the generalized hierarchical product, in Section~\ref{sec:lexico} we proceed with the lexicographic product, while in Section~\ref{sec:corona-and-edge-corona} we consider the corona and the edge corona product. 

Throughout this article, $G=(V_G, E_G)$ stands for a simple graph of order $n(G) = |V_G|$ and size $m(G) = |E_G|$. The degree of $v\in V_G$ is denoted by ${\rm deg}_G(v)$, and the minimum and the maximum degree of $G$ by $\delta_G$ and $\Delta_G$, respectively. If $X\subseteq V_G$, then the subgraph induced by $X$ is denoted by $G\langle V'\rangle$. If $S\subseteq V_G$ and $v\in V_G$, then $N_S(v)$ is the set of neighbors of $v$ in $S$, that is, $N_S(v):=\{u\in S\; | \; uv\in E_G \}$. 
The {\em complement} of $S$ in $V_G$ is denoted by $\bar{S}$. 

Let $X\subseteq V_G$. Then $D\subseteq V_G$ {\em dominates} $X$, if every vertex from $X\setminus D$ has a neighbor in $D$. When $X = V_G$ we say that $D$ is a {\em dominating set} of $G$. The {\em domination number} $\gamma(G)$ is the cardinality of a smallest dominating set of $G$. Let $G=(V_G, E_G)$ be a graph and $k\in \{-\delta_G, \ldots, \delta_G\}$. Then a nonempty set $S\subseteq V_G$ is a {\em global defensive $k$-alliance} in $G$  
if the following conditions hold:
\begin{enumerate}
\item $S$ is a dominating set of $G$ and 
\item for every $v\in S$, $|N_S(v)|\geq |N_{\bar{S}}(v)|+k$. 
\end{enumerate}
When $S$ fulfils Condition 2 it is a {\em defensive $k$-alliance}. To shorten the presentation, we will abbreviate global defensive $k$-alliance as GD$k$-A and defensive $k$-alliance as D$k$-A. Note that GD$k$-A is a D$k$-A which is also a dominating set. 

{\em The global defensive $k$-alliance number} $\gamma^d_k(G)$ of $G$ (abbreviated as {\em GD$k$-A number}) is the smallest order of a GD$k$-A in $G$. If $G$ admits not a single GD$k$-A, we set $\gamma^d_k(G)=\infty$. Similarly, {\em The defensive $k$-alliance number} $\gamma_k(G)$ of $G$ (abbreviated as {\em D$k$-A number}) is the minimum cardinality among all D$k$-As in $G$. If $G$ does not contain a D$k$-A set then we set $\gamma_k(G) = \infty$. Finally, we use the notation $[n] = \{1,\ldots, n\}$. 

\section{Generalized hierarchical products}
\label{sec:hierarchical}

A graph $G$ together with a fixed vertex subset $U\subseteq V_G$ will be denoted by $G(U)$. If $G$ and $H$ are graphs, and $U\subseteq V_G$, then the {\em generalized hierarchical product} $G(U) \sqcap H$ is the graph with the vertex set $V_G \times V_H$, vertices $(g, h)$ and $(g', h')$ being adjacent if and only if either $g = g'\in U$ and $hh' \in E_H$, or  $gg'\in E_G$ and $h = h'$. Note that when $U = V_G$, then the generalized hierarchical product is just the classical Cartesian product of graphs~\cite{hammack-2011}, that is, $G(V_G) \sqcap H = G\cp H$. The generalized hierarchical product was introduced for the first time in~\cite{barriere-2009}, we also refer to \cite{anderson-2017, anderson-2018, arezoomand-2010, tavakoli-2014} for additional results on it as well as on its applications. 

Consider $G(U) \sqcap H$ and let $h\in V(H)$. Then the set of vertices $\{(g,h):\ g\in V_G\}$ is called a {\em $G$-layer over $h$}. Similarly an {\em $H$-layer over $g\in V_G$} is defined. Note that a $G$-layer over $h$ induces a subgraph of $G(U) \sqcap H$ isomorphic to $G$ and an $H$-layer over $g\in U$ induces a subgraph of $G(U) \sqcap H$ isomorphic to $H$. 

\begin{theorem}
\label{thm:hierarchical}
If $G$ and $H$ are graphs and  $U\subseteq V_G$, then 
$$\gamma^d_k(G(U)\sqcap H)\leq \gamma^d_{k}(G)n(H)\,.$$ 
\end{theorem}

\begin{proof}
Let $U\subseteq V_G$, let $S_G$ be a GD$k$-A in $G$ with $|S_G| = \gamma^d_{k}(G)$, and set $S=S_G\times V_H$. 

Note first that $S$ is a domination set of $G(U)\sqcap H$. Indeed, since $S_G$ is a dominating set of $G$, every $G$-layer over $h\in V_H$ is dominated by the intersection of $S$ with the layer. So all $G$-layers are dominated by $S$ and therefore $G(U)\sqcap H$ is dominated by $S$. 

To show that $S$ is a D$k$-A in $G(U)\sqcap H$ consider an arbitrary vertex $(g,h)$ from $S$. 
If $g\in U$, then we have
\begin{align*}
|N_{S}((g,h))|-|N_{\bar{S}}((g,h))|&=(|N_{S_G}(g)| - |N_{\bar{S}_G}(g)|) + |N_{V_H}(h)|\\
& \ge k + \deg_H(h) \\
& \geq k\,,
\end{align*}
and if $g\notin U$, then
$$|N_{S}((g,h))|-|N_{\bar{S}}((g,h))| = (|N_{S_G}(g)| - |N_{\bar{S}_G}(g)|) + 0 \geq k\,.$$
We have thus seen that $S$ is a GD$k$-A. Since $|S| = \gamma^d_{k}(G) n(H)$, the argument is complete. 
\end{proof}

Let $\Gamma$ be the graph of the truncated cube, see the right-hand side of Fig.~\ref{fig:truncated-cube}. Then $\Gamma$ can be represented as the hierarchical product $G(U)\sqcap P_2$, where $G$ is the graph of order $12$ on the left-hand side of  Fig.~\ref{fig:truncated-cube} and $U=\{g_1, g_4, g_9, g_{12}\}$.

\begin{figure}[ht!]
 \centerline{\includegraphics[scale=.4]{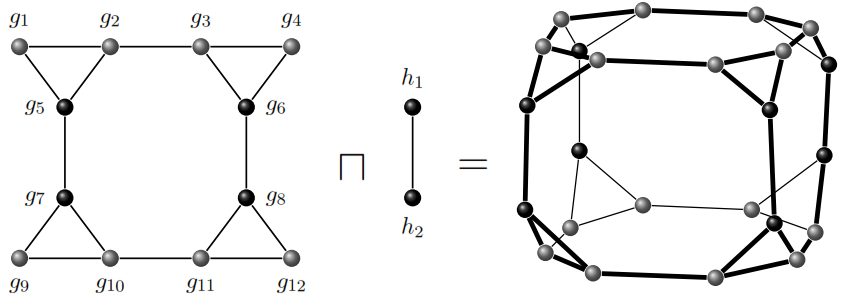}}
\caption{The graph $\Gamma = G(U)\sqcap P_2$, where $U=\{g_1, g_4, g_9, g_{12}\}$.}
\label{fig:truncated-cube}
\end{figure}

The set $S_G=\{g_5, g_6, g_7, g_8\}$ (drawn in black in Fig.~\ref{fig:truncated-cube}) is a global defensive $(-1)$-alliance in $G$. Since $\gamma(G) = 4$ (which follows for instance from the fact that the vertices $g_1, g_4, g_9, g_{12}$ are pairwise at distance at least $3$), we have $\gamma^d_{-1}(G)=4$.  Hence Theorem~\ref{thm:hierarchical} yields $\gamma^d_{-1}(G(U)\sqcap P_2)\leq \gamma^d_{-1}(G) n(P_2) = 8$.  

Now, let $S_{\Gamma}$ be a minimal global defensive $(-1)$-alliance in $\Gamma$. Since $\Gamma$ is a $3$-regular graph and for every $v$ of $V_{\Gamma}$ we have $|N_{S_{\Gamma }}(g)|-|N_{\bar{S}_{\Gamma}}(g)|\geq -1$, each vertex $u$ of $S_{\Gamma}$ must have at least one neighbor in $S_{\Gamma}$. On the other hand, $S_{\Gamma}$ dominates vertices of $\Gamma$ and so $\gamma^d_{-1}(\Gamma)\geq 8$; hence the the inequality of Theorem~\ref{thm:hierarchical} is sharp.

As already mentioned, $G(V_G) \sqcap H = G\cp H$. Hence Theorem~\ref{thm:hierarchical} for the case of the Cartesian product reads as:  
\begin{equation*}
\gamma^d_k(G\cp H)\leq \gamma^d_{k}(G)n(H)\,. 
\end{equation*}
By the well known commutativity property of the Cartesian product operation, this bound further implies that  
$$\gamma^d_{k}(G\cp H)\leq \min \{\gamma^d_{k}(H)n(G), \gamma^d_{k}(G)n(H)\}.$$
The special case of the latter result for $k=1$ has been recently obtained in~\cite{y1}. 

\section{Lexicographic products}
\label{sec:lexico}

The {\em lexicographic product} $G[H]$ of graphs $G$ and $H$ has $V(G[H]) = V_G\times V_H$, vertices $(g_1,h_1)$ and $(g_2,h_2)$ being adjacent if either $g_1g_2\in E_G$, or $g_1=g_2$ and $h_1h_2\in E_H$.  

Let $S$ be a GD$k$-A in $G$ with $k\ge 0$ and set $S = S_G\times V_H$. Since $|N_{S_G}(g)|-|N_{\bar{S}_G}(g)| \ge k$, we get that 
$$|N_{S_G}(g)|n(H) - |N_{\bar{S}_G}(g)|n(H) \ge kn(H)\,,$$
which in turn implies that for any vertex $h\in V_H$, 
$$|N_{S_G}(g)|n(H) + \deg_H(h) - |N_{\bar{S}_G}(g)|n(H) \ge kn(H) + \delta_H\,.$$
Since $|N_{S_G}(g)|n(H) + \deg_H(h) = |N_S(g,h)|$ and $|N_{\bar{S}_G}(g)|n(H) = |N_{\bar{S}}(g,h)|$ this means that $S$ is a global defensive $(kn(H) + \delta_H)$-alliance. As $S$ is also a dominating set, we conclude that 
$$\gamma^d_{kn(H) + \delta_H}(G[H])\le n(H)\gamma^d_k(G)\,.$$
For $k>0$ this is a better result than 
$$\gamma^d_k(G[H])\leq n(H)\gamma^d_k(G)$$
because in such a case $k < kn(H) + \delta_H$, and so $\gamma^d_k(G[H]) \le \gamma^d_{kn(H) + \delta_H}(G[H])$. 

Eballe et al.~\cite{Eballe} obtained some upper bounds of $\gamma^d_k(G[H])$  for the case $k=0$ and $H = K_m$.  In the next theorem, we present an upper bound on  $\gamma^d_k(G[H])$ for the case when there exists a GD$k$-A in $G$ with some special structure. For this sake recall that an {\em $r$-perfect code} in $G = (V_G, E_G)$ is a subset $D$ of $V_G$ for which the balls of radius $r$ centered at the vertices of $D$ form a partition of $V_G$, cf.~\cite{cod}. 

\begin{theorem}
\label{th1}
Let $k > 0$, let $S$ be a smallest GD$k$-A set in $G$, and suppose that $G\langle S\rangle$ has a $1$-perfect code. If $H$ is a graph with more than one vertex, then $G[H]$ has a GD$k$-A. Moreover, if $k \ge 2$, then 
$$\gamma^d_k(G[H])\leq n(H)(\gamma^d_k(G)-\gamma(G\langle S\rangle))
+\gamma(G\langle S\rangle)\,.$$
\end{theorem}

\begin{proof}
First, suppose that $k>2$. Let $S_G$ be a GD$k$-A in $G$ and $D$ be a $1$-perfect code in $G\langle S_G\rangle$. Then $|D| = \gamma(G\langle S_G\rangle)$. (This fact is well known and has been independently established several times, see~\cite[Theorem 9]{haynes-1998}.) 

Set
$S=\big((S_G\setminus D)\times V_H\big)\cup\big(D\times \{v\}\big)$ where $v$ is a vertex of minimum degree in $H$. It is easy to check that $S$ is a dominating set
in $G[H]$. We claim that $S$ is a D$k$-A in $G[H]$. To prove this claim, let $(g,h)\in S$. If $(g,h)\in \big((S_G\setminus D)\times V_H\big)$, then
\begin{align*}
|N_{S}((g,h))|-|N_{\bar{S}}((g,h))|&=
(|N_{S}(g)|-1)n(H)-(|N_{\bar{S}}(g)|+1)n(H)+{\rm deg}_H(h)+2\\
 &=n(H)(|N_{S}(g)|-|N_{\bar{S}}(g)|)-2n(H)+{\rm deg}_H(h)+2\\
 &\geq
(k-2)n(H)+{\rm deg}_H(h)+2\geq k,
\end{align*}
which holds true because $k\ge 2$. 

Consider next a vertex $(g,h)\in D\times \{v\}$. Then 
\begin{align*}
|N_{S}((g,h))|-|N_{\bar{S}}((g,h))|&=
|N_{S}(g)|n(H)-|N_{\bar{S}}(g)|n(H)-\delta_H\\
 &=n(H)(|N_{S}(g)|-|N_{\bar{S}}(g)|)-\delta_H\\
 &\geq
kn(H)-\delta_H \geq k\,,
\end{align*}
which also holds for every $k\ge 2$. 

We have thus proved that $S$ is a D$k$-A in $G[H]$. Moreover, since $S$ contains the copy of $S_G$ in the $G$-layer over $v$, by the definition of the lexicographic product we also infer that $S$ is a dominating set. We conclude that $S$ is a GD$k$-A in $G[H]$. Since clearly $|S| = n(H)(\gamma^d_k(G)-\gamma(G\langle S\rangle)) +\gamma(G\langle S\rangle)$, the proof is complete. 
\end{proof}

The proof of the bound from Theorem~\ref{th1} does not work for $k=1$. The reason is that the inequality $(k-2)n(H)+{\rm deg}_H(h)+2\geq k$ for the case $k=1$ reduces to $-n(H)+{\rm deg}_H(h)+2\geq 1$ which clearly does not hold in general.

\section{Corona and edge corona products}
\label{sec:corona-and-edge-corona}

The {\em corona product} $G\circ H$ of graphs $G$ and $H$ is the graph obtained from the disjoint union of $G$ and $n(G)$ copies of $H$ bijectively assigned to the vertices of $G$, where each vertex $v\in V_G$ is adjacent to all the vertices of the assigned copy of $H$. This product was introduced in~\cite{har}, see also~\cite{tavakoli3, coro}.

\begin{theorem} \label{th4}
If $G$ and $H$ are graphs, then 
$$\gamma^d_k(G\circ H)\leq \min\{n(G)(1 + \gamma^d_{k-1}(H)), n(G)\gamma^d_{k+1}(H)\}\,.$$
Moreover, if $\delta_G-n(H)\geq k$, then $\gamma^d_k(G\circ H)=n(G)$.  
\end{theorem}

\begin{proof}
Let $G'$ be the subgraph of $G\circ H$ isomorphic to $G$ and let $H_i$, $i\in [n(G')]$, be the the isomorphic copy of $H$ corresponding to the $i$th vertex of $G'$. 

Note that $\gamma(G\circ H) = n(G)$ and consequently $\gamma^d_k(G\circ H)\geq n(G)$.
Hence if $\delta_G-n(H)\geq k$, then $V_{G'}$ is a D$k$-A set in $G\circ H$
so that $\gamma^d_k(G\circ H)=n(G)$ holds in this case. 

For the general case, consider an arbitrary defensive $(k-1)$-alliance $S_H$ in $H$. Since each vertex of $H_i$ has exactly one neighbor outside $H_i$, the set $(\cup_{i=1}^{n(G)}S_{H_i})\cup V_{G'}$ is a GD$k$-A in $G\circ H$, where $S_{H_i}$ is
the copy of $S_H$ in $H_i$. Therefore, $\gamma^d_k(G\circ H)\leq n(G)+n(G)\gamma_{k-1}(H) = n(G)(1 + \gamma_{k-1}(H))$.

Also, if $S_H$ is a global defensive $(k+1)$-alliance in $H$, then from the same reasons as above, the set  $\cup_{i=1}^{n(G)}S_{H_i}$  is a GD$k$-A in $G\circ H$. Hence $\gamma^d_k(G\circ H)\leq n(G)\gamma^d_{k+1}(H)$ and therefore, $\gamma^d_k(G\circ H)\leq \min\{n(G)+n(G)\gamma_{k-1}(H), n(G)\gamma^d_{k+1}(H)\}$.
\end{proof}

Consider the corona products $G\circ K_m$, $m\geq 2$. Since $\gamma^d_{1}(K_m)=\left\lceil(m+2)/2\right\rceil$ and  $\gamma^d_{-1}(G)=\left\lfloor (m+1)/2 \right\rfloor$(cf.~\cite{yero-2017}), Theorem~\ref{th4}, yields
 \begin{eqnarray*}
 \gamma^d_{0}(G\circ K_m) & \leq & \min\{n(G)(1 + \gamma^d_{-1}(K_m)), n(G)\gamma^d_{1}(K_m)\}\\
  & = & \min\{n(G)(1+\left\lfloor (m+1)/2\right\rfloor, n(G)\left\lceil (m+2)/2 \right\rceil\}\\
  & = & n(G)\left\lceil (m+2)/2 \right\rceil\,.
\end{eqnarray*}
On the other hand, it was proved in~\cite[Corollary 3.7]{Eballe} that if $m\geq 2$ and $\Delta(G)<m-1$, then $\gamma^d_{0}(G\circ K_m)=n(G)\lceil (m+1)/2\rceil$. It follows that the bound of Theorem~\ref{th4} is sharp for $G\circ K_m$ for all even $m$. 

\medskip
Another corona-like product was recently introduced  as follows. The {\em edge corona} $G\diamondsuit H$ of graphs $G$ and $H$ is obtained by taking one copy of $G$ and $m(G)$ disjoint copies of $H$ associated to the edges of $G$, and for every edge $uv\in E_G$ joining $u$ and $v$ to every vertex of the copy of $H$ associated to $uv$, see~\cite{e-cor1,e-cor2}. For the statement of the next result recall that if $S_G$ is a subset of vertices of a graph $G$, then its complement is denoted with $\bar{S}_G$. 

\begin{theorem}\label{th6}
Let $G$ and $H$ be two graphs. Then
\begin{align*}
\gamma^d_k(G\diamondsuit H)\leq \min\{& m(G)\gamma^d_{k+2}(H),
\gamma^d_{k+n(H)\Delta_G}(G)+\gamma^d_{k+2}(H)|E_{G\langle \bar{S}_G\rangle}|,\\
& \big(\gamma^d_{\lceil\frac{k}{n(H)+1}\rceil}(G)+n(H)|E_{G\langle S'_G\rangle}|+\gamma^d_{k+2}(H)|E_{G\langle \bar{S'}_G\rangle}|\big)I(k)\},
\end{align*}
where $S_G$ is a global defensive $(k+n(H)\Delta_G)$-alliance in $G$, $S'_G$ is a global defensive $\lceil\frac{k}{n(H)+1}\rceil$-alliance in $G$, and 
$$I(k)=\begin{cases}
 1; & \text{if}\; k >-(n(H)+1),\\
 \infty; &  \text{otherwise}\,.
 \end{cases}$$
\end{theorem}

\begin{proof}
Let $G'$ denote the copy of $G$ in $G\diamondsuit H$, and let $H_i$ be the copy of $H$ corresponding to an edge $e_i\in E_{G}$. 

If $S_H$ is a global defensive $(k+2)$-alliance in $H$, then since each vertex of $H_i$ has exactly two neighbors outside $H_i$, the set  $\cup_{i=1}^{m(G)}S_{H_i}$, is a GD$k$-A in $G\diamondsuit H$ (again, $S_{H_i}$ is the copy of $S_H$ in $H_i$).
Thus $\gamma^d_k(G\diamondsuit H)\leq n(G)\gamma^d_{k+2}(H)$.

Let next $S_G$ be a global defensive $(k+n(H)\Delta_{G})$-alliance in $G$, and $S_H$ is a global defensive $(k+2)$-alliance in $H$. Since each vertex of $H_i$ has exactly two neighbors outside $H_i$ and as each vertex of $G'$ has at most $\Delta_{G} n(H)$ neighbors outside $G'$, 
the copy of $S_G$ in $G'$ together with the copies of $S_{H}$ is each of the copies of $H$ corresponding to the edges from $G\langle \bar{S}_G\rangle$ form a global defensive 
$k$-alliance in $G\diamondsuit H$. So, $\gamma^d_k(G\diamondsuit H)\leq \gamma^d_{k+n(H)\Delta_{G}}(G)+\gamma^d_{k+2}(H)|E_{G'\langle \bar{S}_G\rangle}|$.

Suppose now that $S'_G$ is a global defensive $\lceil\frac{k}{n(H)+1}\rceil$-alliance in $G$. Let $S_1$ be the set of vertices of $G\diamondsuit H$ that lie in the copies of $H$ corresponding to the edges of $G'\langle S'_{G'}\rangle$. In addition, in every copy of $H$ corresponding to the edges from $\langle\bar{S'}_{G'}\rangle$ select a global defensive $(k+2)$-alliance in $H$. Then set $S = S'_{G'}\cup S_1 \cup S_2$. We claim that $S$ is a GD$k$-A in $G\diamondsuit H$. 

$S$ is a dominating set in $G\diamondsuit H$, because $S'_{G'}$ dominates $G'$ and all the copies of $H$ above its edges as well as above copies of $H$ above edges with one endpoint in 
$S'_{G'}$, while the other copies of $H$ are dominated by $S_2$. To show that $S$ is a D$k$-A in $G\diamondsuit H$ consider $u\in S'_{G'}$. Then 
\begin{eqnarray*}
|N_{S}(u)|-|N_{\bar{S}}(u)| & = & (|N_{S'_{G}}(u)|-|N_{\bar{S}'_{G}}(u))|\cdot (n(H)+1)\\
& \geq & \left\lceil\frac{k}{n(H)+1}\right\rceil (n(H)+1)\geq k\,.
\end{eqnarray*}
The same conclusion is clear when  $u\in S\setminus S'_{G'}$. Therefore, $S$ is a $k$-alliance in $G\diamondsuit H$ and
so $\gamma^d_k(G\diamondsuit H)\leq \gamma^d_{\lceil\frac{k}{n(H)+1}\rceil}(G)+n(H)|E_{G\langle S'_G\rangle}|+\gamma^d_{k+2}(H)|E_{G\langle \bar{S'}_G\rangle}|$. 
\end{proof}

The {\em sun graph} $S_n$ is obtained by replacing every edge of a cycle $C_n$ by a triangle $C_3$, cf.~\cite{bra}. See Fig.~\ref{fig5} for $S_3$. 

\begin{figure}[ht!] 
 \centerline{\includegraphics[scale=.5]{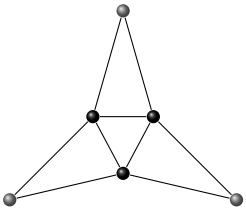}}
\caption{\label{fig5} The sun graph $S_3$}
\end{figure}

From our point of view note that $S_n = C_n\diamondsuit K_1$. Since $\gamma^d_2(K_1)=\infty$, $\gamma^d_2(C_3)=3$, $\gamma^d_0(C_3=2$, $|E_{C_3\langle S_G\rangle}|=1$, $|E_{C_3\langle \bar{S}_G\rangle}|=0$, and $I(0)=1$, Theorem~\ref{th6} implies that $\gamma^d_0(S_3) = \gamma^d_0(C_3\diamondsuit P_1) \leq 3$. Actually, $\gamma^d_0(S_3) = 3$ (in Fig.~\ref{fig5} elements of a global defensive alliance in $C_3\diamondsuit P_1$ are colored black), we have the inequality in Theorem~\ref{th6}.

\medskip

\noindent {\bf Acknowledgements.}
S.K.\ acknowledges the financial support from the Slovenian Research Agency (research core funding No.\ P1-0297 and projects J1-7110, J1-9109, N1-0043). 
The authors are indebted to the referees for helpful remarks which leaded us to correct and improve the paper.

\end{document}